\documentclass[a4paper]{amsart}
\usepackage{graphicx}
\usepackage{float}
\usepackage{amssymb}
\usepackage{amsmath}
\usepackage{mathtools}
\usepackage{bm}
\usepackage{mathrsfs}
\usepackage{extarrows}
\usepackage{tikz-cd}
\usepackage{enumerate} 
\usepackage{cases}
\usepackage{amsfonts}
\usepackage[colorlinks,linkcolor=blue,citecolor=blue]{hyperref}
\usepackage{latexsym, amssymb, amsmath, amsthm, bbm}
\usepackage[all]{xy}

\usepackage{anysize}\marginsize{30mm}{30mm}{30mm}{30mm}
\theoremstyle{plain}
\newtheorem{theorem}{Theorem}[section]
\newtheorem{proposition}[theorem]{Proposition}
\newtheorem{lemma}[theorem]{Lemma}
\newtheorem{corollary}[theorem]{Corollary}

\theoremstyle{definition} 
\newtheorem{definition}[theorem]{Definition}
\newtheorem{remark}[theorem]{Remark}
\newtheorem{example}[theorem]{Example}

\begin{document}

\title{Exact sequences of rt-categories}


\author{Menggao Li}
\address{Department of Mathematics, Nanjing University, Nanjing 210093, China} \email{602022210007@smail.nju.edu.cn}
\author{Gongxiang Liu}
\address{Department of Mathematics, Nanjing University, Nanjing 210093, China} \email{gxliu@nju.edu.cn}

\begin{abstract}
Our aim is to consider what the exact sequence for rt-categories is. For this, we introduce the notion of \emph{exact sequence of rt-categories}, modeled on exact sequences of finite tensor
categories. Our central result explores the relationship of exactness at different levels. Specifically, let $H_1\xrightarrow{f}H_2\xrightarrow{g}H_3$ be a sequence of finite-dimensional Hopf
algebras. We prove that $H_1\xrightarrow{f}H_2\xrightarrow{g}H_3$ is strictly exact if and only if
$H_1\text{-}\mathrm{comod}\xrightarrow{f_*}H_2\text{-}\mathrm{comod} \xrightarrow{g_*}H_3\text{-}\mathrm{comod}$ is an exact sequence of finite tensor categories and $g_*$ admits an exact left
adjoint, if and only if
$D^b_{H_1\text{-}\mathrm{comod}}(H_2\text{-}\mathrm{comod})
\to D^b(H_2\text{-}\mathrm{comod})
\to D^b(H_3\text{-}\mathrm{comod})$
 is an exact sequence of rt-categories and $f_*$ is fully faithful.
\end{abstract}

\maketitle 

\section{Introduction}
The theory of exact sequences of tensor categories was primarily developed by Brugui\`eres and Natale \cite{exact-tensor} as a categorical generalization of the exact sequences of Hopf algebras
\cite{1993Hopf,pri-exact-sequence}. As a foundational paper on exact sequences, it explores a very rich range of content. It reinterprets strictly exact sequences of Hopf algebras and
equivariantization under a finite group action on a tensor category and other notions in the language of exact sequences. Moreover, a finite exact sequence with adjoint functors $G_{l}$ and $G_{r}$ of
the latter functor $G$ can induce a central Hopf algebra $G_{r}(\mathbf{1})$, together with a normal, faithful, $\mathbbm{k}$-linear, right exact Hopf monad $GG_{l}$.

As the main purpose of the paper, we want to build the exact sequence for derived categories of finite tensor categories, or more generally for rt-categories. For this, we recall the definition of
exact sequences for tensor categories of Natale's version \cite{exact-tensor}  at first. Let $\mathcal{A}$, $\mathcal{B}$ and $\mathcal{C}$ be finite tensor categories over $\mathbbm{k}$. A sequence of
tensor functors
\[
\mathcal{A} \xrightarrow{F} \mathcal{B} \xrightarrow{G} \mathcal{C}
\]
is an \emph{exact sequence of finite tensor categories}  
if the following conditions hold:
\begin{enumerate}[\rm (1)]
\item The tensor functor $G$ is surjective and normal;
\item The tensor functor $F$ is fully faithful;
\item $\mathscr{I}_{F}=\mathscr{K}^{\oplus}_{G}$ (see Section 3 for notions).
\end{enumerate}

If we want to give the similar version for rt-categories, a fundamental challenge appears: the classical definitions of \emph{normality} and \emph{surjectivity} for tensor functors are intrinsically
tied to the underlying abelian structure of the categories involved. Specifically, these criteria rely on the well-behaved nature of subobjects, which does not admit a direct analogue in the setting of
monoidal triangulated categories. To better adapt to the rt-category setting, we introduce the following.

\begin{enumerate}
\item \emph{T-Triviality:} We define \emph{t-trivial objects} as objects belonging to the smallest monoidal thick subcategory which is compatible with the given t-structure.
\item \emph{Adjoint Characterization:} We replace the \emph{normality} and \emph{surjectivity} requirements with conditions formulated through \emph{adjoint functors}.
\item \emph{Zero-reflecting:} For monoidal triangulated functors, we use zero-reflecting as the analogue of the faithfulness required for tensor functors.
\end{enumerate}

After making these modifications, we introduce the following definition.
Let $(\mathcal T',\mathbbm t')$, $(\mathcal T,\mathbbm t)$ and $(\mathcal T'',\mathbbm t'')$ be rt-categories. A sequence of t-exact
zero-reflecting monoidal triangulated functors
\[
\mathcal T' \xrightarrow{F} \mathcal T \xrightarrow{G} \mathcal T''
\]
is called an \emph{exact sequence of rt-categories} if the following conditions hold:
\begin{enumerate}
    \item $F$ is fully faithful;
    \item $G$ admits a left adjoint $G_l$ which is zero-reflecting and t-exact,
    and $G_l(\mathbf 1_{\mathcal T''})\in \mathscr K_G^t$;
    \item $\mathscr I_F=\mathscr K_G^t$ (see Section 3 for notions).
\end{enumerate}

Using this, we show that:

\begin{enumerate}
\item \emph{From finite tensor categories to rt-categories}: Take an exact sequence of finite tensor categories $\mathcal{A}'\xrightarrow[]{F} \mathcal{A}\xrightarrow[]{G} \mathcal{A}''$. Then
    $D^b_{\mathcal A'}(\mathcal A) \xrightarrow{F'} D^b(\mathcal A) \xrightarrow{G'} D^b(\mathcal A'')$ is an exact sequence of rt-categories if $G$ admits an exact left adjoint.

\item \emph{From rt-category to finite tensor category}: Take an exact sequence of rt-categories $(\mathcal T',\mathbbm t') \xrightarrow{F} (\mathcal T,\mathbbm t) \xrightarrow{G} (\mathcal
    T'',\mathbbm t'')$. Then $(\mathcal T')^\heartsuit \xrightarrow{F^\heartsuit} \mathcal T^\heartsuit \xrightarrow{G^\heartsuit} (\mathcal T'')^\heartsuit$ is an exact sequence of finite tensor
    categories if the hearts are finite abelian categories.
\end{enumerate}

As a concrete application, we utilize these categorical structures to provide a new criterion for the exactness of sequences of finite-dimensional Hopf algebras, formulated through their comod
categories and derived categories. That is, we get the following main result (see Theorem \ref{thm:core}).

\begin{theorem}
Consider a sequence of finite-dimensional Hopf algebras
\[
K \xrightarrow{f} H \xrightarrow{g} T.
\]
Let
$F=f_*:K\text{-}\mathrm{comod}\to H\text{-}\mathrm{comod}$,
$G=g_*:H\text{-}\mathrm{comod}\to T\text{-}\mathrm{comod}$
be the induced tensor functors. The following are equivalent:
\begin{enumerate}[\rm (1)]
\item The sequence $K\xrightarrow{f}H\xrightarrow{g}T$ is a strictly exact sequence of finite-dimensional Hopf algebras.

\item The sequence
$K\text{-}\mathrm{comod} \xrightarrow{F} H\text{-}\mathrm{comod} \xrightarrow{G} T\text{-}\mathrm{comod}$
is an exact sequence of finite tensor categories, and $G$ admits an exact left adjoint.

\item The functor $F$ is fully faithful, and after identifying $K\text{-}\mathrm{comod}$ with $\mathscr I_F$ as a tensor category, and, with respect to the standard t-structures, the sequence
\[
D^b_{K\text{-}\mathrm{comod}}(H\text{-}\mathrm{comod}) \xrightarrow{\iota} D^b(H\text{-}\mathrm{comod}) \xrightarrow{D^b(G)} D^b(T\text{-}\mathrm{comod})
\]
is an exact sequence of rt-categories, where $\iota$ is the inclusion functor.
\end{enumerate}
\end{theorem}

The paper is organized as follows. In Section 2, we review the preliminaries regarding tensor categories and monoidal triangulated categories. In Section 3, we define exact sequences of rt-categories
and we explore the correspondence between exact sequences of rt-categories and exact sequences of finite tensor categories. Finally, we conclude by demonstrating that these categorical structures
provide an equivalent characterization for the strictly exact sequences of finite-dimensional Hopf algebras.

Throughout this paper, $\mathbbm{k}$ is an algebraically closed field of characteristic zero. All categories are assumed to be $\mathbbm{k}$-linear. For a coalgebra $C$, we denote by
$C\text{-}\operatorname{comod}$ the category of finite-dimensional left $C$-comodules. Unless otherwise stated, all monoidal categories are assumed to
be strict, that is, for all objects $X,Y,Z$ we have $(X\otimes Y)\otimes Z=X\otimes (Y\otimes Z)$ and $X\otimes \mathbf 1=\mathbf 1\otimes X=X$, and the associativity and unit constraints are identity
morphisms. Moreover, subcategories are always assumed to be full subcategories.

\section{Preliminaries}
In this section, we will recall some basic notions and related facts about triangulated categories, tensor categories, in particular about t-structures.

\subsection{Tensor categories and monoidal triangulated categories}
We first recall several definitions and properties related to tensor categories.
Readers are referred to \cite{tensorcategories} for more details.

\begin{definition}\textup{(\cite[Definition 4.1.1]{tensorcategories})}
A rigid monoidal abelian category $(\mathcal A,\otimes,\mathbf 1)$ is called a
\emph{multitensor category} if the bifunctor
$\otimes:\mathcal A\times\mathcal A\to\mathcal A$ is bilinear on morphisms. If,
in addition, $\operatorname{End}_{\mathcal A}(\mathbf 1)\cong\mathbbm k$, then
$\mathcal A$ is called a \emph{tensor category}.
\end{definition}

The tensor product in a multitensor category is exact in each variable
\textup{(see \cite[Proposition 4.2.1]{tensorcategories})}.

\begin{definition}
A \emph{monoidal triangulated category} $(\mathcal T,\otimes,\Sigma,\mathbf 1)$
is a triangulated category equipped with a monoidal structure such that:
\begin{enumerate}[\rm (1)]
\item the shift functor $\Sigma$ satisfies
$\Sigma(X\otimes Y)\cong(\Sigma X)\otimes Y\cong X\otimes(\Sigma Y)$;
\item the bifunctor $\otimes$ is exact in each variable.
\end{enumerate}
We write $X[1]$ for $\Sigma X$. A \emph{monoidal
triangulated functor} is a triangulated functor which preserves the monoidal
structures and the unit object.
\end{definition}

\begin{example}\textup{(\cite[Example 2.4]{m-t})}\label{Ex:D is m-t}
Let $(\mathcal{A},\otimes,\mathbf{1})$ be a monoidal abelian category with biexact tensor product. The category $C^b(\mathcal{A})$ of bounded chain complexes has a monoidal structure defined by the
total complex: for $X^\bullet,Y^\bullet\in C^b(\mathcal{A})$,
\[
(X^\bullet\widetilde{\otimes} Y^\bullet)^n=\bigoplus_{p+q=n}X^p\otimes Y^q
\]
with differential $d^n=\sum_{p+q=n}(d_X^p\otimes\mathrm{id}+(-1)^p\mathrm{id}\otimes d_Y^q)$. This makes $(C^b(\mathcal{A}),\widetilde{\otimes},\mathbf{1}^\bullet)$ a monoidal abelian category, where
$\mathbf{1}^\bullet$ is the stalk complex of $\mathbf{1}$ at degree $0$.
The structure descends to the bounded homotopy category $K^b(\mathcal{A})$ (since $\widetilde{\otimes}$ preserves null-homotopy) and further to the bounded derived category $D^b(\mathcal{A})$ (since
$\widetilde{\otimes}$ preserves quasi-isomorphisms by the Acyclic Assembly Lemma \cite[Lemma 2.7.3]{Weibel_1994}). Thus both $(K^b(\mathcal{A}),\widetilde{\otimes},\mathbf{1}^\bullet)$ and
$(D^b(\mathcal{A}),\widetilde{\otimes},\mathbf{1}^\bullet)$ are monoidal triangulated categories.
\end{example}

\begin{lemma}\label{l01}
For any rigid monoidal abelian category $\mathcal A$ with biexact tensor product, the derived category $(D^b(\mathcal A),\widetilde{\otimes},\mathbf 1^\bullet)$ is rigid.
\end{lemma}

The proof is similar to \cite[Example 2.27]{m-t}.

\begin{definition}
    A functor $F$ is called \emph{zero-reflecting} if, whenever $F(X)$ is a zero object, $X$ itself is a zero object.
\end{definition}

Let $F$ be an exact functor between abelian categories. It is easy to see that $F$ is zero-reflecting if and only if it is faithful. And $D^b(F)$ is zero-reflecting if $F$ is zero-reflecting.

\subsection{Tensor reduced t-structure and heart} We begin with the definition of a t-structure.

\begin{definition}\textup{(\cite[Definition 1.3.1]{Tri-Recollement})}
A pair of full subcategories $\mathbbm{t}=(\mathcal{T}^{\leq 0},\mathcal{T}^{\geq 1})$ in a triangulated category $\mathcal{T}$ is said to be a \emph{t-structure} on $\mathcal{T}$, if
$\mathcal{T}^{\leq 0}$ and $\mathcal{T}^{\geq 1}$ satisfy the following conditions:
\begin{enumerate}[\rm (T1)]
\item $\mathcal{T}^{\leq 0}[1]\subseteq\mathcal{T}^{\leq 0}$ and $\mathcal{T}^{\geq 1}\subseteq\mathcal{T}^{\geq 1}[1]$;
\item $\operatorname{Hom}_{\mathcal{T}}(\mathcal{T}^{\leq 0},\mathcal{T}^{\geq 1})=0$;
\item For any object $X\in\mathcal{T}$, there is a distinguished triangle
\[
X^{\leq 0}\longrightarrow X\longrightarrow X^{\geq 1}\longrightarrow X^{\leq 0}[1],
\]
where $X^{\leq 0}\in\mathcal{T}^{\leq 0}$ and $X^{\geq 1}\in\mathcal{T}^{\geq 1}$.
\end{enumerate}
\end{definition}

For further use, we recall the following standard notation \textup{(cf. \cite[\S 1.3]{Tri-Recollement})}. Let $\mathbbm t_{\mathcal T}=(\mathcal T^{\leq 0},\mathcal T^{\geq 1})$ be a t-structure on
$\mathcal T$. For $n\in\mathbb Z$, set $\mathcal T^{\leq n}:=\mathcal T^{\leq 0}[-n]$ and
$\mathcal T^{\geq n+1}:=\mathcal T^{\geq 1}[-n]$. The \emph{heart} $\mathcal T_{\mathbbm t}^{\heartsuit}:= \mathcal T^{\leq 0}\cap\mathcal T^{\geq 0}$ is an abelian category, and there is a cohomology
functor $H^0_{\mathbbm t}:\mathcal T\to\mathcal T_{\mathbbm t}^{\heartsuit}$. We write $\mathcal T^\heartsuit$ for $\mathcal T_{\mathbbm t}^{\heartsuit}$ and $H^0$ for $H^0_{\mathbbm t}$ when no
confusion arises. The t-structure $\mathbbm t_{\mathcal T}$ is called \emph{bounded} if, for every object $X\in\mathcal T$, one has $X[n]\in\mathcal T^{\leq 0}$ for all $n\gg0$ and $X[n]\in\mathcal
T^{\geq 1}$ for all $n\ll0$.

\begin{definition}\textup{(\cite[Definition 2.10]{m-t})}
A bounded t-structure $\mathbbm t=(\mathcal T^{\leq 0},\mathcal T^{\geq 1})$ on a monoidal triangulated category $\mathcal T$ is called a \emph{monoidal t-structure} if there exists $n\in\mathbb Z$
such that \begin{enumerate}[\rm (1)]
\item $\mathcal T^{\geq 0}\otimes\mathcal T^{\geq n} \subseteq\mathcal T^{\geq n}$;
\item $\mathcal T^{\leq 0}\otimes\mathcal T^{\leq n} \subseteq\mathcal T^{\leq n}$.
\end{enumerate}
The set of integers $n$ satisfying conditions \textup{(1)} and \textup{(2)} is called the \emph{deviation} of $\mathbbm t$ and is denoted by $\operatorname{dev}(\mathbbm t)$.
\end{definition}

A monoidal triangulated category equipped with a monoidal t-structure $\mathbbm t$ such that $0\in\operatorname{dev}(\mathbbm t)$ is called a \emph{monoidal t-category}. If, in addition, the underlying
monoidal triangulated category is rigid and $\operatorname{End}_{\mathcal T}(\mathbf 1)\cong\mathbbm k$, we call it a \emph{tensor t-category}.

\begin{definition}\textup{(\cite[Definition 2.18]{m-t})}
A monoidal additive category $\mathcal A$ is called \emph{tensor reduced} if, for any object $X\in\mathcal A$, one has $X\otimes X=0$ if and only if $X=0$. A monoidal t-structure $\mathbbm t$ on a
monoidal triangulated category $\mathcal T$ with $0\in\operatorname{dev}(\mathbbm t)$ is called
\emph{tensor reduced} if its heart $\mathcal T^\heartsuit$ is tensor reduced as a monoidal additive category.
\end{definition}

\begin{corollary}\textup{(\cite[Proposition 2.25, Corollary 2.26]{m-t})}\label{heart}
Let $\mathbbm t$ be a tensor reduced monoidal t-structure on a monoidal triangulated category $\mathcal T$.
\begin{enumerate}[\rm (1)]
\item $H^0_{\mathbbm t}(\mathbf 1)\cong\mathbf 1$ in $\mathcal T_{\mathbbm t}^\heartsuit$.
\item If $\mathcal T$ has left duals, then $(\mathcal T^{\leq 0})^*\subseteq\mathcal T^{\geq 0}$ and $(\mathcal T^{\geq 0})^*\subseteq\mathcal T^{\leq 0}$. In particular, $\mathcal T^\heartsuit$ has
    left duals.
\item If $\mathcal T$ has right duals, then ${}^*(\mathcal T^{\leq 0})\subseteq\mathcal T^{\geq 0}$ and ${}^*(\mathcal T^{\geq 0})\subseteq\mathcal T^{\leq 0}$. In particular, $\mathcal T^\heartsuit$
    has right duals.
\item if $\mathcal T$ is a tensor t-category, then $\mathcal T^\heartsuit$ is a tensor category.
\end{enumerate}
\end{corollary}

We say that a tensor t-category is \emph{rt-category} if it is tensor reduced.

\begin{example}
Let $\mathcal{A}$ be an abelian category. The \emph{standard t-structure}
(or \emph{canonical t-structure}) on $D^b(\mathcal{A})$ is defined by:
\begin{align*}
{D}^{\leq 0}(\mathcal{A}) &= \{X \in D^b(\mathcal{A}) \mid H^i(X) = 0 \text{ for all } i > 0\}, \\
{D}^{\geq 1}(\mathcal{A}) &= \{X \in D^b(\mathcal{A}) \mid H^i(X) = 0 \text{ for all } i < 1\}.
\end{align*}
The heart ${\mathcal{A}}^\heartsuit = {D}^{\leq 0}(\mathcal{A}) \cap {D}^{\geq 0}(\mathcal{A})$ is naturally equivalent to $\mathcal{A}$.

In particular, when $\mathcal{A}$ is a finite tensor category, the unit object $\mathbf{1}$ is simple. Hence $X$ is a subobject of $X\otimes X\otimes X^{*}$ via the monomorphism $id_{X}\otimes
coev_{X}$. So this standard t-structure is tensor reduced.
\end{example}

A triangulated functor between categories with t-structures is called
left, right, or fully t-exact in the usual sense. We shall use the standard fact
\cite[Proposition 1.3.17]{Tri-Recollement} that a t-exact functor restricts to
an exact functor on the hearts.

\begin{definition}\textup{(\cite[\S 1.3.16]{Tri-Recollement})}
Let $\mathcal{T}$ and $\mathcal{T'}$ be triangulated categories with t-structures. A triangulated functor $F:\mathcal{T}\to\mathcal{T'}$ is called \emph{t-exact} if $F(\mathcal{T}^{\geq
0})\subseteq\mathcal{T'}^{\geq 0}$ and $F(\mathcal{T}^{\leq 0})\subseteq\mathcal{T'}^{\leq 0}$.
\end{definition}

By \cite[Proposition 1.3.17]{Tri-Recollement}, it follows that if $F$ is t-exact, then $F^\heartsuit=F|_{\mathcal{T}^\heartsuit}$ is exact.

\subsection{Thick subcategory compatible with t-structure} Let $\mathcal{T}$ be a triangulated category with a bounded t-structure $\mathbbm{t}=(\mathcal{T}^{\leq 0}, \mathcal{T}^{\geq 1})$ and
$\mathcal{I}$ a thick subcategory. Let $Q:\mathcal{T}\to \mathcal{T}/\mathcal{I}$ be the quotient functor.

\begin{definition}\textup{(\cite[Definition 3.5]{perverse})}
$\mathcal{I}$ is said to be compatible with $\mathbbm{t}$ if $\mathbbm{t}_{\mathcal{T}/\mathcal{I}} := (Q(\mathcal{T}^{\leq 0}), Q(\mathcal{T}^{\geq 1}))$ is a t-structure on $\mathcal{T}/\mathcal{I}$.
\end{definition}


\begin{lemma}\textup{(\cite[Lemma 3.10]{perverse})}\label{lem:serre-thick}
There is a bijection
\[
\{\text{thick subcategories of } \mathcal{T} \text{ compatible with }
\mathbbm{t}\} \longleftrightarrow \{\text{Serre subcategories of } \mathcal{T}^{\heartsuit}\}
\]
given by $\mathcal{I} \mapsto \mathcal{I} \cap \mathcal{T}^{\heartsuit}$ and
$\mathcal{J} \mapsto \mathcal{T}_\mathcal{J} = \{X \in \mathcal{T} \mid
H^i(X) \in \mathcal{J}, \forall i \in\mathbb{Z}\}$.
\end{lemma}

\section{Exact sequences of rt-categories}
In this section, we introduce the definition of exact sequences of rt-categories and explores the relationship of exactness at different levels.

\subsection{The definition of exact sequences of rt-categories}
In this subsection, we introduce exact sequences of rt-categories. The definition is modeled on the classical exact sequences of finite tensor categories.

We first recall the classical notions from \cite{exact-tensor}. Let $\mathcal A$ be a finite tensor category. An object of $\mathcal A$ is called \emph{trivial} if it is isomorphic to a finite direct
sum of copies of the unit object. For a tensor functor $F:\mathcal A\to\mathcal A'$, its kernel is the full subcategory
\[
\mathscr K_F^\oplus = \{X\in\mathcal A\mid F(X)\text{ is a trivial object of }\mathcal A'\}.
\]
We also denote by $\mathscr I_F$ the essential image of $F$. A sequence of tensor functors
$\mathcal A' \xrightarrow{F} \mathcal A \xrightarrow{G} \mathcal A''$
between finite tensor categories is an exact sequence if $F$ is fully faithful, $G$ is surjective and normal, and $\mathscr I_F=\mathscr K_G^\oplus$. We now pass to the rt-categorical analogue.

\begin{definition}
Let $(\mathcal T,\mathbbm t)$ be an rt-category. We denote by $\mathcal V^t_{\mathcal T}$ the smallest monoidal thick subcategory of
$\mathcal T$ which closed under duals and is compatible with $\mathbbm t$. Objects of $\mathcal V^t_{\mathcal T}$ are called \emph{t-trivial objects}.
\end{definition}

\begin{definition}
Let $F:\mathcal T\to\mathcal T'$ be a monoidal triangulated functor between rt-categories. The \emph{t-kernel} of $F$ is the full subcategory
\[
\mathscr K_F^t = \{X\in\mathcal T\mid F(X)\text{ is a t-trivial object of }\mathcal T'\}.
\]
The \emph{essential image} of $F$ is the full subcategory
\[
\mathscr I_F = \{X\in\mathcal T'\mid X\cong F(T)\text{ for some }T\in\mathcal T\}.
\]
\end{definition}

\begin{definition}
Let $(\mathcal T',\mathbbm t')$, $(\mathcal T,\mathbbm t)$ and $(\mathcal T'',\mathbbm t'')$ be rt-categories. A sequence of t-exact
zero-reflecting monoidal triangulated functors
\[
\mathcal T' \xrightarrow{F} \mathcal T \xrightarrow{G} \mathcal T''
\]
is called an \emph{exact sequence of rt-categories} if the following conditions hold:
\begin{enumerate}
    \item $F$ is fully faithful;
    \item $G$ admits a left adjoint $G_l$ which is zero-reflecting and t-exact,
    and $G_l(\mathbf 1_{\mathcal T''})\in \mathscr K_G^t$;
    \item $\mathscr I_F=\mathscr K_G^t$.
\end{enumerate}
\end{definition}

\subsection{Exact sequences from bounded derived categories}
We now relate this definition to bounded derived categories of finite tensor categories. Let $(\mathcal A,\otimes,\mathbf 1)$ be a finite tensor category and set $\mathcal D:=D^b(\mathcal A)$. Equipped
with the standard t-structure, $\mathcal D$ is an rt-category. We first describe the t-trivial objects in $\mathcal D$.

\begin{lemma}\label{lem:D^b_A(B)}
Let $\mathcal A$ be a monoidal abelian category with biexact tensor product, and let $\mathcal S\subseteq \mathcal A$ be a monoidal Serre subcategory. Set
\[
D^b_{\mathcal S}(\mathcal A):=\{X\in D^b(\mathcal A)\mid H^i(X)\in\mathcal S \text{ for all }i\in\mathbb Z\}.
\]
Then $D^b_{\mathcal S}(\mathcal A)$ is the smallest monoidal thick subcategory of $D^b(\mathcal A)$ containing $D^b(\mathcal S)$.
If $\mathcal A$ is rigid and $\mathcal S$ is closed under duals, then $D^b_{\mathcal S}(\mathcal A)$ is closed under duals.
\end{lemma}

\begin{proof}
In fact, $D^b_{\mathcal S}(\mathcal A)$ is a thick subcategory of $D^b(\mathcal A)$ containing $D^b(\mathcal S)$.

We next prove the minimality. Let $\mathcal I\subseteq D^b(\mathcal A)$ be a thick subcategory containing $D^b(\mathcal S)$. By the standard truncation triangles \cite[Proposition
13.1.15]{triangulated-categories}, every object $X^\bullet\in D^b_{\mathcal S}(\mathcal A)$ lies in the thick subcategory generated by the objects $H^i(X^\bullet)[-i]$. Since $H^i(X^\bullet)\in\mathcal
S$, we have $H^i(X^\bullet)[-i]\in\mathcal I$ for all $i$, and therefore $X^\bullet\in\mathcal I$. Thus $D^b_{\mathcal S}(\mathcal A)$ is the smallest thick subcategory of $D^b(\mathcal A)$ containing
$D^b(\mathcal S)$.

It remains to check that this thick subcategory is monoidal. Let $X^\bullet,Y^\bullet\in D^b_{\mathcal S}(\mathcal A)$. By
\cite[Theorem 4.1]{BK-tt-cate}, for all $n\in\mathbb Z$ we have
\[
H^n(X^\bullet\widetilde{\otimes} Y^\bullet) \cong \bigoplus_{p+q=n}H^p(X^\bullet)\otimes H^q(Y^\bullet).
\]
Since $\mathcal S$ is a monoidal Serre subcategory, each $H^p(X^\bullet)\otimes H^q(Y^\bullet)$ belongs to $\mathcal S$, and hence the
finite direct sum above also belongs to $\mathcal S$. Thus $X^\bullet\otimes Y^\bullet\in D^b_{\mathcal S}(\mathcal A)$. Moreover
$\mathbf 1\in\mathcal S$, so the stalk complex $\mathbf 1^\bullet$ belongs to $D^b_{\mathcal S}(\mathcal A)$. Therefore $D^b_{\mathcal S}(\mathcal A)$ is a monoidal thick subcategory. The minimality
proved above implies that $D^b_{\mathcal S}(\mathcal A)$ is the smallest monoidal thick subcategory of $D^b(\mathcal A)$ containing $D^b(\mathcal S)$.

Finally assume that $\mathcal A$ is rigid and that $\mathcal S$ is closed under duals. We prove the claim for left duals; the proof for right duals is similar. Let $X^\bullet\in D^b_{\mathcal
S}(\mathcal A)$, and let $Z^\bullet$ be the left dual complex of $X^\bullet$, so that $Z^{-i}=(X^i)^*$ with the usual signed dual differential. By \cite[Proposition 4.2.9]{tensorcategories}, left
dualization is exact. Hence the signs in the dual differential do not affect kernels, images, or cokernels, and a direct computation gives $H^{-i}(Z^\bullet)\cong H^i(X^\bullet)^*$ for all $i\in\mathbb
Z$. Indeed, this follows by dualizing the short exact sequence
\[
0\longrightarrow H^i(X^\bullet) \longrightarrow \operatorname{coker}(d_X^{i-1}) \longrightarrow \operatorname{im}(d_X^i) \longrightarrow 0.
\]
Since $H^i(X^\bullet)\in\mathcal S$ and $\mathcal S$ is closed under duals, we have $H^i(X^\bullet)^*\in\mathcal S$. Therefore
$H^{-i}(Z^\bullet)\in\mathcal S$ for all $i$, so $Z^\bullet\in D^b_{\mathcal S}(\mathcal A)$. Thus $D^b_{\mathcal S}(\mathcal A)$ is closed under left duals, and similarly it is closed under right
duals.
\end{proof}

\begin{lemma}\label{lem:rigidserre-rigidthick}
Let $\mathcal A$ be a monoidal abelian category with biexact tensor product, and let $\mathbbm t$ be the standard t-structure on $D^b(\mathcal A)$. We regard $\mathcal A\simeq D^b(\mathcal
A)^\heartsuit$ as a full subcategory of $D^b(\mathcal A)$ via stalk complexes. Then there is a bijection between monoidal thick subcategories of $D^b(\mathcal A)$ compatible with $\mathbbm t$ and
monoidal Serre subcategories of $\mathcal A$, given by
\[
\mathcal I\longmapsto \mathcal I\cap\mathcal A, \qquad \mathcal J\longmapsto D^b_{\mathcal J}(\mathcal A).
\]
Moreover, if $\mathcal A$ is rigid, then this bijection restricts to a bijection between $\mathbbm t$-compatible dual-closed monoidal thick subcategories of $D^b(\mathcal A)$ and dual-closed monoidal Serre subcategories of $\mathcal A$.
\end{lemma}

\begin{proof}
By Lemma~\ref{lem:serre-thick}, the assignments
\[
\mathcal I\longmapsto \mathcal I\cap\mathcal A, \qquad \mathcal J\longmapsto D^b_{\mathcal J}(\mathcal A)
\]
are mutually inverse bijections between thick subcategories of $D^b(\mathcal A)$ compatible with $\mathbbm t$ and Serre subcategories of
$\mathcal A$. It remains to check that these bijections preserve the monoidal condition and the condition of being closed under duals.

Let $\mathcal J$ be a monoidal Serre subcategory of $\mathcal A$. Then Lemma~\ref{lem:D^b_A(B)} implies that $D^b_{\mathcal J}(\mathcal A)$ is a
monoidal thick subcategory of $D^b(\mathcal A)$. If, moreover, $\mathcal A$ is rigid and $\mathcal J$ is closed under duals, then the same lemma shows that $D^b_{\mathcal J}(\mathcal A)$ is closed
under duals.
Conversely, let $\mathcal I$ be a monoidal thick subcategory of $D^b(\mathcal A)$ compatible with $\mathbbm t$. Then $\mathcal I\cap\mathcal A$ is a Serre subcategory of $\mathcal A$ by
Lemma~\ref{lem:serre-thick}. Since $\mathcal I$ is monoidal, it contains $\mathbf 1^\bullet$. Hence $\mathbf 1\in\mathcal I\cap\mathcal A$. Moreover, if $X,Y\in\mathcal I\cap\mathcal A$, then,
regarding them as stalk complexes, we have $X\widetilde{\otimes} Y\in\mathcal I$; and since $X\widetilde{\otimes} Y$ is again an object of $\mathcal A$, it follows that $X\widetilde{\otimes}
Y\in\mathcal I\cap\mathcal A$. Thus $\mathcal I\cap\mathcal A$ is a monoidal Serre subcategory of $\mathcal A$.

If $\mathcal A$ is rigid and $\mathcal I$ is closed under duals, then for every $X\in\mathcal I\cap\mathcal A$, its left and right duals in $\mathcal A$, regarded as stalk complexes in $D^b(\mathcal
A)$, belong to $\mathcal I$ and still lie in $\mathcal A$. Hence $\mathcal I\cap\mathcal A$ is closed under duals.
\end{proof}

Let $\mathcal V^\oplus\subseteq\mathcal A$ be the trivial subcategory of $\mathcal{A}$. Since $\mathbbm k$ has characteristic $0$ and $\mathcal A$ is a finite tensor category, every object of $\mathcal
A$ has finite length, the unit object $\mathbf 1$ is simple, and $\operatorname{Ext}^1(\mathbf 1,\mathbf 1)=0$ \textup{(see \cite[Theorem 4.3.8 and Theorem 4.4.1]{tensorcategories})}.
It follows that $\mathcal V^\oplus$ is a Serre subcategory of $\mathcal A$. Indeed, subobjects and quotients of finite direct sums of copies of $\mathbf 1$ are again finite direct sums of copies of
$\mathbf 1$, and every extension of such objects splits. Moreover, $\mathcal V^\oplus$ is monoidal. Since every Serre subcategory of $\mathcal A$ containing $\mathbf 1$ contains all finite direct sums
of copies of $\mathbf 1$, $\mathcal V^\oplus$ is the smallest monoidal Serre subcategory of $\mathcal A$.

Therefore, by Lemma~\ref{lem:rigidserre-rigidthick}, the t-trivial subcategory of $\mathcal D$ is $D^b_{\mathcal V^\oplus}(\mathcal A)$. Thus
\[
\mathcal V^t_\mathcal{D} = \{X^\bullet\in\mathcal D \mid H^i(X^\bullet)\cong \mathbf 1^{\oplus n_i} \text{ for some } n_i\geq 0 \text{ and for all } i\in\mathbb Z\}.
\]
We now show that exact sequences of finite tensor categories give rise to exact sequences of rt-categories after passing to bounded derived categories.

\begin{proposition}\label{prop:exact-seq-derived}
Let
\[
\mathcal A' \xrightarrow{F} \mathcal A \xrightarrow{G} \mathcal A''
\]
be an exact sequence of finite tensor categories over $\mathbbm k$. Assume that $G$ admits an exact left adjoint functor $G_l$. Then
\[
D^b_{\mathcal A'}(\mathcal A) \xrightarrow{F'} D^b(\mathcal A) \xrightarrow{G'} D^b(\mathcal A'')
\]
is an exact sequence of rt-categories, where $F'$ is the inclusion functor and $G'=D^b(G)$ is the functor induced by $G$.
\end{proposition}

\begin{proof}
By the exactness of the sequence of finite tensor categories, we identify $\mathcal A'$ with $\mathscr I_F=\mathscr K_G^\oplus$. In particular, $\mathcal A'$ is a monoidal Serre subcategory of
$\mathcal A$ and is closed under duals. Hence, by Lemma~\ref{lem:rigidserre-rigidthick}, $D^b_{\mathcal A'}(\mathcal A)$, equipped with the induced standard t-structure, is an rt-category.

The functor $F'$ is the inclusion functor, so it is fully faithful, monoidal, triangulated, t-exact and zero-reflecting. Since $G$ is an exact tensor functor, it induces a monoidal triangulated functor
$G'=D^b(G):D^b(\mathcal A)\to D^b(\mathcal A'')$, and $G'$ is t-exact for the standard t-structures. Moreover, since tensor functors between finite tensor categories are faithful, $G'$ is
zero-reflecting.

Since $G_l$ is exact, it induces a t-exact triangulated functor $G_l'=D^b(G_l):D^b(\mathcal A'')\to D^b(\mathcal A)$. By \cite[Proposition 2.3.5]{Dfunctor}, the adjunction $G_l\dashv G$ induces an
adjunction $G_l'\dashv G'$. Since $G$ is surjective, by \cite[Proposition 5.1]{exact-tensor}, the left adjoint $G_l$ is faithful. Thus $G_l'$ is zero-reflecting.
By the normality of $G$ and \cite[Proposition 3.5]{exact-tensor}, we have $G_l(\mathbf 1_{\mathcal A''})\in\mathscr K_G^\oplus$. Then the stalk complex $G_l(\mathbf 1_{\mathcal
A''})^\bullet=G_l'(\mathbf 1_{\mathcal A''}^\bullet)$ belongs to $\mathscr K_{G'}^t$. Thus the second condition in the definition of an exact sequence of rt-categories is satisfied.

It remains to prove $\mathscr I_{F'}=\mathscr K_{G'}^t$. Since $F'$ is the inclusion functor, $\mathscr I_{F'}=D^b_{\mathcal A'}(\mathcal A)$. For $X^\bullet\in D^b(\mathcal A)$, the exactness of $G$
gives $H^i(G'(X^\bullet))\cong G(H^i(X^\bullet))$ for all $i$. Hence $X^\bullet\in\mathscr K_{G'}^t$ if and only if $G(H^i(X^\bullet))$ is ordinary trivial for all $i$, equivalently
$H^i(X^\bullet)\in\mathscr K_G^\oplus$ for all $i$. Since the original sequence is exact, $\mathscr K_G^\oplus=\mathscr I_F=\mathcal A'$ under the above identification. Therefore $X^\bullet\in\mathscr
K_{G'}^t$ if and only if $X^\bullet\in D^b_{\mathcal A'}(\mathcal A)$. Thus $\mathscr K_{G'}^t=D^b_{\mathcal A'}(\mathcal A)=\mathscr I_{F'}$.
\end{proof}

\subsection{Exact sequences from hearts}

The preceding proposition shows that exact sequences of finite tensor categories give rise to exact sequences of rt-categories by passing to bounded derived categories. We now discuss the converse
operation, namely taking hearts. This will show that the derived construction is compatible with the original exact sequence.

\begin{proposition}\label{prop:heart-exact-seq}
Let
\[
(\mathcal T',\mathbbm t') \xrightarrow{F} (\mathcal T,\mathbbm t) \xrightarrow{G} (\mathcal T'',\mathbbm t'')
\]
be an exact sequence of rt-categories.
Assume that the hearts $(\mathcal T')^\heartsuit$, $\mathcal T^\heartsuit$ and $(\mathcal T'')^\heartsuit$ are finite abelian categories.
Then the induced sequence on hearts
\[
(\mathcal T')^\heartsuit \xrightarrow{F^\heartsuit} \mathcal T^\heartsuit \xrightarrow{G^\heartsuit} (\mathcal T'')^\heartsuit
\]
is an exact sequence of finite tensor categories and $G^\heartsuit$ admits an exact left adjoint functor.
\end{proposition}

\begin{proof}
By Proposition~\ref{heart} and the assumption that the hearts are finite abelian categories, the three hearts are finite tensor categories. Since $F$ and $G$ are t-exact monoidal triangulated functors,
they restrict to exact tensor functors $F^\heartsuit$ and $G^\heartsuit$ on the hearts. Moreover, $F^\heartsuit$ is fully faithful because $F$ is fully faithful.

Let $G_l$ be the t-exact zero-reflecting left adjoint of $G$. Since both $G_l$ and $G$ are t-exact, the adjunction $G_l\dashv G$ restricts to an adjunction $G_l^\heartsuit\dashv G^\heartsuit$ on the
hearts. The functor $G_l^\heartsuit$ is exact and zero-reflecting, hence faithful. Therefore, by \cite[Proposition 5.1]{exact-tensor}, $G^\heartsuit$ is surjective.

We next prove normality. By definition, $G_l(\mathbf 1_{\mathcal T''})\in\mathscr K_G^t$. Since $G_l$ is t-exact, $G_l(\mathbf 1_{\mathcal T''})$ lies in $\mathcal T^\heartsuit$. Moreover,
$G^\heartsuit G_l^\heartsuit(\mathbf 1_{(\mathcal T'')^\heartsuit})$ is a
t-trivial object lying in $(\mathcal T'')^\heartsuit$, hence lie in the smallest monoidal Serre subcategory of $(\mathcal T'')^\heartsuit$, that is, the ordinary trivial subcategory of $(\mathcal
T'')^\heartsuit$. Thus $G_l^\heartsuit(\mathbf 1_{(\mathcal T'')^\heartsuit})$ belongs to $\mathscr K^\oplus_{G^\heartsuit}$, and so $G^\heartsuit$ is normal by \cite[Proposition 3.5]{exact-tensor}.

It remains to identify the kernel. Since $F$ is t-exact and fully faithful, we have $\mathscr I_{F^\heartsuit}=\mathscr I_F\cap\mathcal T^\heartsuit$. Also, for $X\in\mathcal T^\heartsuit$, the object
$X$ belongs to $\mathscr K_G^t$ if and only if $G^\heartsuit(X)$ is ordinary trivial. Hence $\mathscr K_G^t\cap\mathcal T^\heartsuit=\mathscr K^\oplus_{G^\heartsuit}$. Using $\mathscr I_F=\mathscr
K_G^t$, we obtain
\[
\mathscr I_{F^\heartsuit} = \mathscr I_F\cap\mathcal T^\heartsuit = \mathscr K_G^t\cap\mathcal T^\heartsuit = \mathscr K^\oplus_{G^\heartsuit}.
\]
Thus the induced sequence on hearts is an exact sequence of finite tensor categories, and $G^\heartsuit$ admits the exact left adjoint
$G_l^\heartsuit$.
\end{proof}

To specify precisely in what sense the original sequence is recovered after passing to derived categories and then taking hearts, we use the following notion of equivalence of sequences.

\begin{definition}
Let
$\mathcal A' \xrightarrow{F} \mathcal A \xrightarrow{G} \mathcal A''$ and $\mathcal B' \xrightarrow{\widetilde F} \mathcal B \xrightarrow{\widetilde G} \mathcal B''$ be two sequences of tensor
categories. We say that these two sequences are \emph{equivalent} if there exist tensor equivalences  $\Phi':\mathcal A'\to\mathcal B'$, $ \Phi:\mathcal A\to\mathcal B$, and $ \Phi'':\mathcal
A''\to\mathcal B''$ together with monoidal natural isomorphisms $\Phi\circ F\cong \widetilde F\circ \Phi'$ and $ \Phi''\circ G\cong \widetilde G\circ \Phi$.
\end{definition}

\begin{corollary}\label{cor:derived-heart-recovers}
With the notation and assumptions of Proposition~\ref{prop:exact-seq-derived}, the sequence obtained by taking the hearts of
\[
D^b_{\mathcal A'}(\mathcal A) \xrightarrow{F'} D^b(\mathcal A) \xrightarrow{G'} D^b(\mathcal A'')
\]
with respect to the standard t-structures is equivalent to the original exact sequence
\[
\mathcal A' \xrightarrow{F} \mathcal A \xrightarrow{G} \mathcal A''.
\]
\end{corollary}

\begin{proof}
The hearts of $D^b(\mathcal A)$ and $D^b(\mathcal A'')$ with respect to the standard t-structures are canonically equivalent to $\mathcal A$ and
$\mathcal A''$, respectively. Moreover, by the convention $D^b_{\mathcal A'}(\mathcal A)=D^b_{\mathscr I_F}(\mathcal A)$, the heart of $D^b_{\mathcal A'}(\mathcal A)$ is canonically equivalent to
$\mathscr I_F$. Since $F$ is fully faithful, $F$ induces a tensor equivalence $\mathcal A'\simeq \mathscr I_F$.
Under these equivalences, the functor induced by $F'$ on hearts is identified with $F$, and the functor induced by $G'=D^b(G)$ on hearts is identified with $G$. Hence the sequence obtained by taking
hearts is equivalent to the original sequence of finite tensor categories.
\end{proof}

\subsection{Exact sequences of finite-dimensional Hopf algebras}

In this subsection, we prove our main result. Let us recall the definition of strictly exact sequences of Hopf algebras.

\begin{definition}\textup{(\cite{pri-exact-sequence})}\label{def:exact-hopf}
A \emph{strictly exact sequence of Hopf algebras} is a diagram
\[
K \stackrel{i}{\longrightarrow} H \stackrel{p}{\longrightarrow} T
\]
where $i$ and $p$ are Hopf algebra morphisms such that
\begin{enumerate}[\rm (1)]
\item $K$ is a normal Hopf subalgebra of $H$;
\item $H$ is right faithfully flat over $K$;
\item $p$ is a categorical cokernel of $i$,
\end{enumerate}
or, equivalently, setting $I=p^{-1}(0)$, such that
\begin{enumerate}[\rm (1)]
\item $I$ is a normal Hopf ideal of $H$;
\item $H$ is right faithful coflat over $T$;
\item $i$ is a categorical kernel of $p$,
\end{enumerate}
where $\ker(f)=\{x\in H\mid x_{(1)}\otimes f(x_{(2)})\otimes x_{(3)}=x_{(1)}\otimes 1\otimes x_{(2)}\}$ and $\operatorname{coker}(f)=T/Tf(H^{+})T$.
\end{definition}

In \cite{exact-tensor}, the author briefly introduced its relationship with exact sequences of tensor categories.

\begin{remark}\textup{(\cite[Remark 3.13]{exact-tensor})}\label{rem:tensor-to-hopf}
Let $\mathcal{C}' \xrightarrow{f} \mathcal{C} \xrightarrow{F} \mathcal{C}''$ be an exact sequence of tensor categories over $\mathbbm{k}$ such that $F$ admits an exact left adjoint. Assume that
$\mathcal{C}''$ admits a fiber functor $\omega: \mathcal{C}'' \to \operatorname{Vec}_{\mathbbm{k}}$. Setting
\[
H' = \operatorname{L}(\omega), \quad H = \operatorname{L}(\omega F), \quad K = \operatorname{L}(\omega F f),
\]
and denoting by $i: K \to H$ and $p: H \to H'$ the Hopf algebra morphisms induced by $f$ and $F$, respectively, we obtain a strictly exact sequence of Hopf algebras over $\mathbbm{k}$:
$K \xrightarrow{i} H \xrightarrow{p} H'$.

\end{remark}

Finally, we obtain the following theorem.

\begin{theorem}\label{thm:core}
Consider a sequence of finite-dimensional Hopf algebras
\[
K \xrightarrow{f} H \xrightarrow{g} T.
\]
Let
$F=f_*:K\text{-}\mathrm{comod}\to H\text{-}\mathrm{comod}$,
$G=g_*:H\text{-}\mathrm{comod}\to T\text{-}\mathrm{comod}$
be the induced tensor functors. The following are equivalent:
\begin{enumerate}[\rm (1)]
\item The sequence $K\xrightarrow{f}H\xrightarrow{g}T$ is a strictly exact sequence of finite-dimensional Hopf algebras.

\item The sequence
$K\text{-}\mathrm{comod} \xrightarrow{F} H\text{-}\mathrm{comod} \xrightarrow{G} T\text{-}\mathrm{comod}$
is an exact sequence of finite tensor categories, and $G$ admits an exact left adjoint.

\item The functor $F$ is fully faithful, and after identifying $K\text{-}\mathrm{comod}$ with $\mathscr I_F$ as a tensor category, and, with respect to the standard t-structures, the sequence
\[
D^b_{K\text{-}\mathrm{comod}}(H\text{-}\mathrm{comod}) \xrightarrow{\iota} D^b(H\text{-}\mathrm{comod}) \xrightarrow{D^b(G)} D^b(T\text{-}\mathrm{comod})
\]
is an exact sequence of rt-categories, where $\iota$ is the inclusion functor.
\end{enumerate}
\end{theorem}

\begin{proof}
\textup{(1)} $\Longleftrightarrow$ \textup{(2)}:
Let $g:H\to T$ be a Hopf algebra morphism. The induced corestriction functor
$G=g_*:H\text{-}\mathrm{comod}\longrightarrow T\text{-}\mathrm{comod}$
admits a right adjoint given by the cotensor product $G_r=H\Box_T -$. For $V\in T\text{-}\mathrm{comod}$, the cotensor product $H\Box_TV$ is defined as the equalizer
\[
H\Box_TV \longrightarrow H\otimes V \rightrightarrows H\otimes T\otimes V,
\]
where the two arrows are induced by the right $T$-coaction on $H$ and the left $T$-coaction on $V$.

If $K \xrightarrow{f} H \xrightarrow{g} T$ is strictly exact, then $H$ is faithfully coflat as a right $T$-comodule. Hence $H\Box_T-$ is exact and zero-reflecting. By rigidity, the left adjoint of $G$
is obtained from the right adjoint by duality, namely $G_l(X)\cong {}^*G_r(X^*)$ for $X\in T\text{-}\mathrm{comod}$. Since duality is an exact equivalence, $G_l$ is also exact and zero-reflecting.
Therefore the sequence
\[
K\text{-}\mathrm{comod} \xrightarrow{F} H\text{-}\mathrm{comod} \xrightarrow{G} T\text{-}\mathrm{comod}
\]
is an exact sequence of finite tensor categories (see \cite[\S 3]{exact-tensor}), and $G$ admits an exact left adjoint.

Conversely, assume \textup{(2)}. Choosing the fiber functor $U:T\text{-}\mathrm{comod}\to\operatorname{Vec}_{\mathbbm k}$, the reconstruction theorem \textup{\cite[Section 5]{tensorcategories}} yields
a commutative diagram
\[
\begin{tikzcd}
K \arrow{d}[swap]{\cong} \arrow{r}{f} &
H \arrow{d}{\cong} \arrow{r}{g} &
T \arrow{d}{\cong} \\
L(UGF) \arrow{r}[swap]{i} &
L(UG) \arrow{r}[swap]{p} &
L(U).
\end{tikzcd}
\]
The bottom row is strictly exact by Remark \ref{rem:tensor-to-hopf}, hence so is the top row.
Thus \textup{(1)} and \textup{(2)} are equivalent.

\textup{(2)} $\Longrightarrow$ \textup{(3)}:
This follows directly from Proposition~\ref{prop:exact-seq-derived}.

\textup{(3)} $\Longrightarrow$ \textup{(2)}:
Taking hearts with respect to the standard t-structures and applying Proposition~\ref{prop:heart-exact-seq}, we obtain an exact sequence of finite tensor categories
\[
\big(D^b_{K\text{-}\mathrm{comod}}(H\text{-}\mathrm{comod})\big)^\heartsuit \to H\text{-}\mathrm{comod}\to T\text{-}\mathrm{comod},
\]
and the last functor admits an exact left adjoint. Since
$\big(D^b_{K\text{-}\mathrm{comod}}(H\text{-}\mathrm{comod})\big)^\heartsuit \simeq \mathscr I_F$
and $F$ is fully faithful, this sequence is equivalent to
$K\text{-}\mathrm{comod}\xrightarrow{F}H\text{-}\mathrm{comod} \xrightarrow{G}T\text{-}\mathrm{comod}$.
Thus \textup{(2)} holds.
\end{proof}

\end{document}